\documentclass[reqno,12pt]{amsart}
\usepackage{amssymb,amsmath,enumerate}
\usepackage{tikz,hyperref}
\usepackage{mathrsfs, bbm}

\usetikzlibrary{shapes.geometric}
\tikzstyle{poly}=[regular polygon, draw, line width=2, blue!60]
\tikzstyle{minipoly}=[regular polygon, draw, line width=1.2, blue!60]
\tikzstyle{v}=[circle, inner sep=1.5, fill=white, draw=gray, thick]
\tikzstyle{v0}=[circle, inner sep=1, fill=white, draw=gray]

\tikzstyle{path}=[draw, line width=1.2, color=black]
\tikzstyle{pathlight}=[draw, line width=1, dotted, color=lightgray]

\numberwithin{equation}{section}
\newtheorem{theorem}[equation]{Theorem}

\newtheorem{corollary}[equation]{Corollary}
\newtheorem{lemma}[equation]{Lemma}

\newcommand{\drawbranch}[3]{%
	\def\x{#1}
	\def\y{#2}
	\def\sf{#3}
	\draw[scale=\sf,very thick,blue] (\x,\y) node[dot]{} -- (\x-.6,\y-.8) node[dot]{}; 
	\draw[scale=\sf,very thick,blue] (\x,\y) -- (\x+.6,\y-.8) node[dot]{};
}

\def\B{\mathcal{B}}
\def\M{\mathcal{M}}
\def\ST{\mathcal{ST}}
\def\SP{\mathcal{SP}^\wedge}
\def\S{\mathcal{S}}
\def\T{\mathcal{T}}
\def\Y{\mathscr{Y}}

\def\D{\mathsf{D}}
\def\E{\mathsf{E}}
\def\N {\mathsf{N}}
\def\R{\rule[-1ex]{0ex}{3.6ex}}

\title{Schr\"oder Coloring and Applications}
\author{Daniel Birmajer}
\address{Nazareth College\\ 4245 East Ave.\\ Rochester, NY 14618}
\author{Juan B. Gil}
\address{Penn State Altoona\\ 3000 Ivyside Park\\ Altoona, PA 16601}
\author{Juan D. Gil}
\address{Massachusetts Institute of Technology\\ Cambridge, MA 02139}
\author{Michael D. Weiner}
\address{Penn State Altoona\\ 3000 Ivyside Park\\ Altoona, PA 16601}

\begin{document}
\maketitle

\begin{abstract}
We present several bijections, in terms of combinatorial objects counted by the Schr\"oder numbers, that are then used (via coloring) for the construction and enumeration of rational Schr\"oder paths with integer slope, ordered rooted trees, and simple rooted outerplanar maps. On the other hand, we derive partial Bell polynomial identities for the little and large Schr\"oder numbers, which allow us to obtain explicit enumeration formulas.
\end{abstract}

\section{Introduction}
 
This paper focuses on the little Schr\"oder numbers
\begin{equation} \label{eq:Schroeder} 
	1, 1, 3, 11, 45, 197, 903, 4279, 20793, 103049,\dots,
\end{equation}
and the role they play for coloring certain combinatorial structures. Consistent with the OEIS, we denote the above sequence by $(s_n)_{n\in\mathbb{N}_0}$. These numbers appeared in a paper by Ernst Schr\"oder \cite{Schr1870} from 1870, where he discusses four combinatorial problems regarding parenthesizations. Problem II in \cite{Schr1870} lead to \eqref{eq:Schroeder}, and it is now well known that $(s_n)_{n\in\mathbb{N}_0}$ enumerates a wide variety of objects, see \cite[A001003]{oeis}, \cite{Stan97}. 

Here we rely on the following three interpretations for $s_n$ with $n\ge 0$:

\begin{itemize}
\item[({\sc path})] Number of Schr\"oder paths\footnote{Lattice paths with steps $\E=(1,0)$, $\N=(0,1)$, and $\D=(1,1)$, staying weakly above $y=x$.} from $(0,0)$ to $(n,n)$ with no $\D$-steps on $y=x$.
\item[({\sc tree})] Number of ordered trees with no vertex of outdegree 1 and having $n+1$ leaves. 
\item[({\sc poly})] Number of dissections of a convex $(n+2)$-gon by nonintersecting diagonals. 
\end{itemize}

Our main computational result is Theorem~\ref{thm:BellLittleSchroeder}, where we give an explicit formula for the evaluation of partial Bell polynomials at $s_0, s_1, s_2, \dots$. Our proof relies on a combinatorial argument that uses integer compositions and ({\sc path}). In addition, we give a similar formula for when $(s_n)_{n\in\mathbb{N}_0}$ is replaced by the sequence $(r_n)_{n\in\mathbb{N}_0}$ of large Schr\"oder numbers $1, 2, 6, 22, 90,\dots$, see \cite[A006318]{oeis}. These formulas are particularly beneficial when dealing with Bell transformations, as defined in \cite{BGW18}. 

In Section~\ref{sec:FussSchroeder}, we consider the set $\S_n(\alpha)$ of rational Schr\"oder paths with slope $\alpha\in\mathbb{N}$, i.e., lattice paths from $(0,0)$ to $(n,\alpha n)$ with steps $\E$, $\N$, and $\D$, that stay weakly above the line $y=\alpha x$. Our main contribution here is a bijective map that provides an algorithm to generate rational Schr\"oder paths using classical Schr\"oder paths as building blocks. As a consequence, we obtain a formula for the enumeration of $\S_n(\alpha)$ by the number of building blocks used for their construction. In the special case when these paths contain the maximal number of blocks of the form $\D$ or $\N\E$, our formula leads to several sequences listed in the OEIS \cite{oeis}, see Table~\ref{tab:paths_k=n}.

Section~\ref{sec:SchroederTrees} deals with ordered rooted trees. We consider ordered trees with no vertex of outdegree 1, like in ({\sc tree}), and use them to construct ordered rooted trees with a given number of generators.\footnote{A {\em generator} is a leaf or a node with only one child.} This is achieved by means of a Dyck path insertion algorithm described in the proof of Theorem~\ref{thm:dyck2tree}. As a corollary, we obtain a formula for the number of ordered rooted trees with $n$ generators having a prescribed number of nodes of outdegree 1.

Finally, in Section~\ref{sec:OuterplanarMaps}, we make a brief excursion into the enumeration of outerplanar maps, cf.\ \cite{BGH05, GN17}. We view polygon dissections as biconnected simple outerplanar maps and use ({\sc poly}), together with certain colored Dyck paths, to generate and count simple outerplanar maps by the number of their biconnected components.

All of our proofs rely on explicit combinatorial bijections that involve building blocks counted by the Schr\"oder numbers. The three applications we discussed in this paper represent just a glimpse of the possible ways Schr\"oder-type objects may be used to construct and enumerate more complex combinatorial classes. For example, $s_n$ also counts indecomposable permutations that avoid the patterns $2413$ and $3142$, and for $n\ge 1$, $s_n$ gives the number of increasing tableaux of shape $(n,n)$, cf.\ \cite{oeis}. 

In summary, the content of Section~\ref{sec:BellSchroeder} provides a toolbox for studying a broad family of sequence transformations of the little and large Schr\"oder numbers.

\section{Bell transforms of little Schr\"oder numbers}
\label{sec:BellSchroeder}

In \cite{BGW18}, three of the authors introduced the following family of sequence transformations defined via partial Bell polynomials. Let $a$, $b$, $c$, $d$ be fixed numbers. Given $x=(x_n)_{n\in\mathbb{N}}$, its Bell transform $y=\Y_{a,b,c,d}(x)$ is the sequence defined by
\begin{equation} \label{eq:transform}
   y_n=\sum_{k=1}^n  \frac{1}{n!}
   \bigg[\prod_{j=1}^{k-1}{(an+bk+cj+d)}\bigg] B_{n,k}(1!x_1,2!x_2,\dots) \text{ for } n\ge 1,
\end{equation}
where $B_{n,k}$ denotes the $(n, k)$-th (exponential) partial Bell polynomial. 

For $0\le k\le n$, we have
\begin{equation*}
 B_{n,k}(z_1,\dots, z_{n-k+1})
      = \!\!\sum_{\alpha\in\pi(n,k)}\frac{n!}{\alpha_1!\alpha_2!\cdots}\left(\frac{z_1}{1!}\right)^{\alpha_1}
      \left(\frac{z_2}{2!}\right)^{\alpha_2}\cdots,
\end{equation*}
where $\pi(n,k)$ denotes the set of multi-indices $\alpha\in\mathbb{N}_0^{n-k+1}$ such that
\[  \alpha_1+\alpha_2+\dots=k \;\text{ and }\; \alpha_1+2\alpha_2+3\alpha_3+\dots=n. \]
The polynomial $B_{n,k}$ contains as many monomials as the number of partitions of $[n]=\{1,\dots,n\}$ into $k$ parts. 

The definition of $\Y_{a,b,c,d}(x)$ suggests the need for finding a closed formula for $B_{n,k}(1!x_1,2!x_2,\dots)$, which in most cases is a challenging task. In this section, we will use a combinatorial argument (via colored compositions) to derive a formula for $B_{n,k}(1!s_0,2!s_1,\dots)$, where $(s_n)_{n\in\mathbb{N}_0}$ is the sequence of little Schr\"oder numbers.

Let $\mathcal{CS}_{n,k}$ be the set of compositions of $n$ with exactly $k$ parts such that part $j$ can take on $s_{j-1}$ colors. Using a formula\footnote{Written here in terms of Bell polynomials.} by Hoggatt and Lind \cite{HL68}, we get

\begin{equation} \label{eq:schroederCompositions}
  \#\mathcal{CS}_{n,k} = \frac{k!}{n!} B_{n,k}(1!s_0,2!s_1,3!s_2,\dots).
\end{equation}

\begin{lemma} \label{lem:schroederComp2Paths}
There is a bijection between the elements of $\mathcal{CS}_{n+k,k}$ and the set of Schr\"oder paths from $(0,0)$ to $(n+k-1,n+k-1)$ with exactly $k-1$ diagonal steps on the line $y=x$.  
\end{lemma}
\begin{proof}
Given a Schr\"oder path $P$ from $(0,0)$ to $(n+k-1,n+k-1)$ with $k-1$ diagonal steps on the line $y=x$, use these diagonal steps to split the paths into $k$ Schr\"oder subpaths $P_1,P_2,\dots,P_k$ with no $\D$-steps on the line $y=x$, including the trivial path of length 0 (i.e., a single lattice point). For example, for $n=5$ and $k=4$, we consider the path

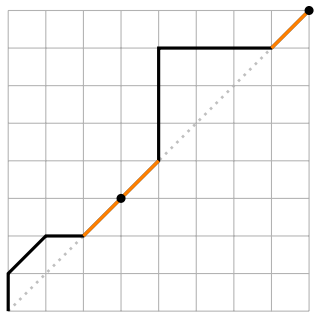
\begin{center}
\tikz[scale=0.5]{\draw[pathlight] (0,0) -- (8,8);
	\draw[step=1,gray!60] (0,0) grid (8,8); 
	\draw[path] (0,0) -- (0,1) -- (1,2) -- (2,2) -- (4,4) -- (4,7) -- (7,7) -- (8,8);
	\foreach \i in {2,3,7} {\draw[path,orange] (\i,\i) -- (\i+1,\i+1);}
	\foreach \i in {3,8} {\fill (\i,\i) circle(0.12);}}
\end{center}

\noindent
and use the three diagonal steps (in orange) to split the path into the four subpaths in black. The first path covers a segment on $y=x$ with 3 lattice points, and the third path covers a segment with 4 lattice points. The second and fourth paths are trivial. With the above path, we then associate the composition $9 = 3+1+4+1$, where part 3 is labeled by the path $\N\D\E$ and part 4 is labeled by the path $\N\N\N\E\E\E$.

In general, if $P_j$ covers a segment on $y=x$ with $i_j$ lattice points, we let $\varphi(P_j)=i_j$ and define $\varphi(P)$ to be the composition
\[ n+k = i_1+i_2+\dots+i_k, \]
where part $i_j$ is labeled by $P_j$. The map $\varphi$ is clearly bijective.
\end{proof}

As a consequence of Lemma~\ref{lem:schroederComp2Paths}, ({\sc path}), and identity \eqref{eq:schroederCompositions}, we get
\begin{equation} \label{eq:schroederConvolution}
  \frac{k!}{(n+k)!} B_{n+k,k}(1!s_0,2!s_1,\dots) = \sum\limits_{m_1+\dots+m_k=n} s_{m_1}\cdots s_{m_k}.
\end{equation}

\begin{theorem} \label{thm:BellLittleSchroeder}
For $1\le k<n$, we have
\begin{equation}
  \label{eq:schroederBell}
  B_{n,k}(1!s_0,2!s_1,3!s_2,\dots) 
  = \frac{n!}{(k-1)!} \sum_{j=1}^{n-k}\frac{1}{j}\binom{n-k-1}{j-1}\binom{n+j-1}{j-1}.
\end{equation}
\end{theorem}

\begin{proof}
First of all, note that $s_0=1$, and for $n\ge 1$,
\begin{align*}
 s_n &= \sum_{j=1}^{n} \frac{1}{j} \binom{n+j}{j-1}\binom{n-1}{j-1} \\
 &= \sum_{j=1}^{n}\binom{n+j}{j-1}\frac{(j-1)!}{n!} \frac{n!}{j!}\binom{n-1}{j-1} \\
 &= \sum_{j=1}^{n}\binom{n+j}{j-1}\frac{(j-1)!}{n!} B_{n,j}(1!,2!,\dots).
\end{align*}
Therefore, by a convolution formula given in \cite[Section 4]{BGW18},
\begin{align*}
  \sum_{m_1+\dots+m_k=n} \!\! s_{m_1}\cdots s_{m_k}
  & = k\sum_{j=1}^{n}\binom{n+j+k-1}{j-1}\frac{(j-1)!}{n!} B_{n,j}(1!,2!,\dots) \\
  & = \sum_{j=1}^{n}\frac{k}{j}\binom{n-1}{j-1}\binom{n+j+k-1}{j-1}.
\end{align*}
Now, identity \eqref{eq:schroederConvolution} implies
\[ \frac{k!}{(n+k)!} B_{n+k,k}(1!s_0,2!s_1,\dots) 
   = \sum_{j=1}^{n}\frac{k}{j}\binom{n-1}{j-1}\binom{n+j+k-1}{j-1}, \]
and \eqref{eq:schroederBell} follows by replacing $n$ by $n-k$.
\end{proof}

Using basic properties of the partial Bell polynomials and \eqref{eq:schroederBell}, we get:
\begin{corollary}
If $(r_n)_{n\in\mathbb{N}_0}$ is the sequence of large Schr\"oder numbers $1, 2, 6, 22, \dots$, then for $1\le k<n$, we have
\begin{align*}
  B_{n,k}(1!r_0,& \,2!r_1,3!r_2,\dots) \\
  &= \sum_{\ell=0}^k (-1)^\ell 2^{k-\ell}\binom{n}{\ell} B_{n-\ell,k-\ell}(1!s_0,2!s_1,3!s_2,\dots) \\
  &= \frac{n!}{(k-1)!} \sum_{\ell=0}^k  \sum_{j=1}^{n-k} \frac{(-1)^\ell 2^{k-\ell}}{j}
  \binom{k-1}{\ell}\binom{n-k-1}{j-1}\binom{n-\ell+j-1}{j-1}.
\end{align*}
\end{corollary}

It is worth noting that some of the most prominent Bell transforms are of the form $\Y_{a,b,-1,1}(x)$ with $a,b\in \mathbb{N}_0$. In such a special case, if $s=(s_n)_{n\in\mathbb{N}_0}$ and $y=\Y_{a,b,-1,1}(s)$, then \eqref{eq:transform} becomes
\begin{equation*}
   y_n=\sum_{k=1}^n  \binom{an+bk}{k-1} \frac{(k-1)!}{n!} B_{n,k}(1!s_0,2!s_1,\dots).
\end{equation*}

We can then use \eqref{eq:schroederBell} to conclude:
\begin{corollary}
If $s=(s_n)_{n\in\mathbb{N}_0}$, then the Bell transform $\Y_{a,b,-1,1}(s)$ is given by
\begin{equation*}
   y_n= \frac{1}{n}\binom{(a+b)n}{n-1} 
   + \sum_{k=1}^{n-1}\sum_{j=1}^{n-k}\frac{1}{j} \binom{an+bk}{k-1}\binom{n-k-1}{j-1}\binom{n+j-1}{j-1}.
\end{equation*}
\end{corollary}

\section{Rational Schr\"oder paths}
\label{sec:FussSchroeder}

For $a\in\mathbb{N}$ and $\mathbf{c}=(c_1,c_2,\dots)$ with $c_j\in\mathbb{N}_0$, let us consider the set of Dyck words of semilength $(a+1)n$ created from strings of the form $P_0=``d\,"$ and $P_j=``u^{(a+1)j}d^{j}"$ for $j=1,\ldots,n$, such that each maximal $(a+1)j$-ascent substring $u^{(a+1)j}$ may be colored in $c_j$ different ways. Consistent with the notation used in \cite{BGMW}, we denote this set by $\mathfrak{D}^{\mathbf{c}}_n(a,1)$.

Let $\SP_n$ denote the set of Schr\"oder paths from $(0,0)$ to $(n,n)$ that lie {\em strictly} above the diagonal $y = x$ except for the diagonal path from $(0,0)$ to $(1,1)$. Moreover, let $\pmb{\sigma}$ be the sequence that enumerates $\SP_n$. In particular, $\sigma_1=2$ since the two paths $\D$ and $\N\E$ from $(0,0)$ to $(1,1)$ both fit the above description, and $\sigma_2, \sigma_3, \sigma_4,\dots$, is precisely the sequence of large Schr\"oder numbers $2, 6, 22, 90, 394,\dots$, see \cite[A006318]{oeis}. Note that $\sigma_j = 2s_{j-1}$ for $j\ge 1$.

Recall that $\S_n(\alpha)$ denotes the set of rational Schr\"oder paths with slope $\alpha\in\mathbb{N}$.

\begin{theorem} 
For every $n\in\mathbb{N}$, we have
\begin{equation*}
 \#\S_n(\alpha) = \#\mathfrak{D}^{\pmb{\sigma}}_n(\alpha-1,1).
\end{equation*}
\end{theorem}

\begin{proof}
We will give an explicit bijective map $\xi:\S_n(\alpha)\to \mathfrak{D}^{\pmb{\sigma}}_n(\alpha-1,1)$. 

Given a rational Schr\"oder path $p\in \S_n(\alpha)$, walk the path from $(n, \alpha n)$ to $(0,0)$ and construct a colored Dyck path $q\in\mathfrak{D}^{\pmb{\sigma}}_n(\alpha-1,1)$ as follows:
\begin{itemize}
\item if at a point $(a,b)$ on $p$ followed by a $\D$-step, add an $\alpha$-ascent followed by one down-step to the path $q$, color the ascent by \tikz[baseline=3pt,scale=0.4]{\draw[gray!60] (0,0) rectangle (1,1); \draw[path,blue] (0,0) -- (1,1);}, and move to the point $(a-1,b-1)$ on $p$.
\item if at a point $(a,b)$ followed by an $\E$-step, draw a line of slope 1 from that point to the next point of the form $(a-j,b-j)$ where the line intersects the path $p$. The path $p_j$ between these two points is a Schr\"oder path in $\SP_j$. Add to the path $q$ an $\alpha j$-ascent followed by $j$ down-steps, color that ascent by $p_j$, and move to the point $(a-j,b-j)$ on $p$.
\item if at a point $(a,b)$ followed by an $\N$-step, add a down-step to the path $q$ and move to the point $(a,b-1)$ on $p$. 
\end{itemize}
The process stops when $(0,0)$ has been reached. The path $q=\xi(p)$ created by means of the above algorithm is by construction a Dyck path in $\mathfrak{D}^{\pmb{\sigma}}_n(\alpha-1,1)$. Note that a return to the line $y=\alpha x$ on $p$ corresponds to a return to the $x$-axis on the Dyck path. For an example, see Figure~\ref{fig:schroeder2dyck}.

\begin{figure}[ht]
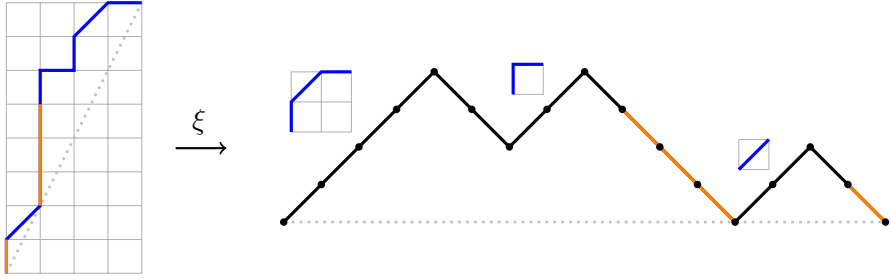

\tikz[scale=0.45]{
 \draw[step=1,gray!60] (0,0) grid (4,8); \draw[pathlight] (0,0) -- (4,8);
 \draw[path,blue] (0,0) -- (0,1) -- (1,2) -- (1,6) -- (2,6) -- (2,7) -- (3,8) -- (4,8);
 \draw[path,orange] (0,0) -- (0,1); \draw[path,orange] (1,2) -- (1,5);
 \draw[->,thick] (5,3.7) -- (6.5,3.7);
 \node[above=1pt] at (5.7,3.7) {\small $\xi$};
}
\tikzstyle{dot}=[circle, inner sep=1, fill=black]
\quad \tikz[scale=0.5,baseline=-20pt]{
 \draw[pathlight] (0,0) -- (16,0);
 \draw[path] (0,0) -- (4,4) -- (6,2) -- (8,4) -- (12, 0) -- (14,2) -- (16,0);
 \draw[path,orange] (9,3) -- (12,0); \draw[path,orange] (15,1) -- (16,0);
 \foreach \i in {0,...,4} {\node[dot] at (\i,\i) {};}
 \foreach \i/\j in {5/3,6/2,7/3,8/4,9/3,10/2,11/1,12/0,13/1,14/2,15/1,16/0} {\node[dot] at (\i,\j) {};}
 \node[above=10pt] at (1,0) 
 {\tikz[scale=0.4]{\draw[gray!60] (0,0) grid (2,2); \draw[path,blue] (0,0) -- (0,1) -- (1,2) -- (2,2);}};
 \node[above=10pt] at (6.5,1) 
 {\tikz[scale=0.4]{\draw[gray!60] (0,0) rectangle (1,1); \draw[path,blue] (0,0) -- (0,1) -- (1,1);}};
 \node[above=10pt] at (12.5,-1) 
 {\tikz[scale=0.4]{\draw[gray!60] (0,0) rectangle (1,1); \draw[path,blue] (0,0) -- (1,1);}}
}
\caption{Example for $\alpha=2$ and $n=4$.}
\label{fig:schroeder2dyck}
\end{figure}

The map $\xi$ is reversible. Given a Dyck path $q\in \mathfrak{D}^{\pmb{\sigma}}_n(\alpha-1,1)$, construct a path $p\in\S_n(\alpha)$ as follows. Walk the Dyck path $q$ from right to left. For every down-step that is not part of a block $P_j=u^{\alpha j}d^{j}$ on $q$, add an $\N$-step to $p$, and for every block $P_j$, add to $p$ the Schr\"oder path labeling the corresponding $\alpha j$-ascent on $P_j$.
\end{proof}

The above bijection provides an algorithm to generate rational Schr\"oder paths in $\S_n(\alpha)$ using classical Schr\"oder paths as building blocks.

By \cite[Theorem~2]{BGMW}, the number of elements in $\mathfrak{D}^{\pmb{\sigma}}_n(\alpha-1,1)$ having exactly $k$ peaks is given by
\begin{align*}
  &\binom{(\alpha-1)n+k}{k-1} \frac{(k-1)!}{n!} B_{n,k}(1!\sigma_1,2!\sigma_2,\dots) \\
  &=\binom{(\alpha-1)n+k}{k-1} \frac{(k-1)!}{n!} B_{n,k}(1!(2s_0),2!(2s_1),\dots) \\
  &=\binom{(\alpha-1)n+k}{k-1} \frac{(k-1)!}{n!} 2^k B_{n,k}(1!s_0,2!s_1,\dots).
\end{align*}

\begin{corollary}
There are 
\begin{equation} \label{eq:n_peaks}
 \frac{2^n}{n}\binom{\alpha n}{n-1} = \frac{2^n}{(\alpha-1)n+1} \binom{\alpha n}{n}
\end{equation}
paths in $\S_n(\alpha)$ containing $n$ paths from $\SP_1$ (i.e., paths of the form $\D$ or $\N\E$).
\end{corollary}
Table~\ref{tab:paths_k=n} shows a few terms of the sequence \eqref{eq:n_peaks} for $\alpha=1,\dots,5$. 

\begin{table}[ht]
\small
\begin{tabular}{|c|l|c|} \hline
\rule[-2ex]{0ex}{5.2ex} Slope $\alpha$ &\hfill $\frac{2^n}{(\alpha-1)n+1} \binom{\alpha n}{n}$\hfill \ & OEIS \\[2pt] \hline
\R 1 & 2, 4, 8, 16, 32, 64, 128, 256, 512, 1024, \dots & A000079 \\ \hline
\R 2 & 2, 8, 40, 224, 1344, 8448, 54912, 366080, \dots & A151374 \\ \hline
\R 3 & 2, 12, 96, 880, 8736, 91392, 992256, 11075328, \dots & A153231 \\ \hline
\R 4 & 2, 16, 176, 2240, 31008, 453376, 6888960, 107707392, \dots & A217360 \\ \hline
\R 5 & 2, 20, 280, 4560, 80960, 1520064, 29680640, 596593920, \dots & A217364 \\ \hline
\end{tabular}
\bigskip
\caption{Number of paths in $\S_n(\alpha)$ having $n$ subpaths from $\SP_1$.}
\label{tab:paths_k=n}
\end{table}

\begin{corollary}
If $1\le k<n$, the number of paths in $\S_n(\alpha)$ that can be built with exactly $k$ elements of $\SP = \bigcup_n \SP_n$ is given by
\begin{equation*}
   2^k\binom{(\alpha-1)n +k}{k-1}\sum_{j=1}^{n-k} \frac{1}{j}\binom{n-k-1}{j-1}\binom{n+j-1}{j-1}.
\end{equation*}
\end{corollary}

Finally, for the cardinality of $\S_n(\alpha)$, we can use \cite[Theorem~2.9]{Schr07} to deduce the simple formula
\begin{equation*}
   \#\\S_n(\alpha) = \frac{1}{\alpha n+1} \sum_{\ell=0}^n \binom{\alpha n+1}{n-\ell}\binom{\alpha n+\ell}{\ell}.
\end{equation*}

\section{Ordered rooted trees}
\label{sec:SchroederTrees}

In this section, we use the interpretation ({\sc tree}) to generate and count ordered rooted trees with $n$ generators (leaves or nodes of outdegree 1).

Let $\ST_{\!n}$ be the set of ordered trees with no vertex of outdegree 1 and having $n+1$ leaves. Thus $\#\ST_{\!n} = s_n$. Note that every tree in $\ST_{\!n}$ has $n+1$ generators.

\begin{table}[ht]
\tikzstyle{tree}=[circle, draw, fill, inner sep=0pt, minimum width=2.5pt]
\def\sf{0.45}
\small
\begin{tabular}{c|c|c}
\R \quad$n=0$\quad\ & \quad$n=1$\quad\ & $n=2$ \\ \hline
\tikz[scale=\sf, baseline=0]{\node[tree] at (0,0) {};} 
& \tikz[scale=\sf, baseline=0]{\draw (0,0) node[tree]{} -- (-1,-1) node[tree]{}; \draw (0,0) -- (1,-1) node[tree]{};} 
& \;\tikz[scale=\sf+.07, baseline=0]{\draw (0,0) node[tree]{} -- (-1,-1) node[tree]{}; \draw (0,0) -- (0,-1) node[tree]{}; \draw (0,0) -- (1,-1) node[tree]{};}  
\quad
\tikz[scale=\sf, baseline=0]{\draw (0,0) node[tree]{} -- (-1,-1) node[tree]{}; \draw (0,0) -- (1,-1) node[tree]{};
\draw (-1,-1) -- (-2,-2) node[tree]{}; \draw (-1,-1) -- (0,-2) node[tree]{};} 
\quad
\tikz[scale=\sf, baseline=0]{\draw (0,0) node[tree]{} -- (-1,-1) node[tree]{}; \draw (0,0) -- (1,-1) node[tree]{};
\draw (1,-1) -- (0,-2) node[tree]{}; \draw (1,-1) -- (2,-2) node[tree]{};} \\[30pt] \hline\hline
\multicolumn{3}{c}{\R $n=3$} \\ \hline
\multicolumn{3}{c}{
\tikz[scale=\sf+.1, baseline=0]{\draw (0,0) node[tree]{} -- (-1.5,-1) node[tree]{}; \draw (0,0) -- (-0.5,-1) node[tree]{}; 
\draw (0,0) -- (0.5,-1) node[tree]{}; \draw (0,0) -- (1.5,-1) node[tree]{};}  
\quad
\tikz[scale=\sf, baseline=0]{\draw (0,0) node[tree]{} -- (-1,-1) node[tree]{}; \draw (0,0) -- (1,-1) node[tree]{};
\draw (-1,-1) node[tree]{} -- (-2,-2) node[tree]{}; \draw (-1,-1) -- (-1,-2) node[tree]{}; \draw (-1,-1) -- (0,-2) node[tree]{};} 
\quad
\tikz[scale=\sf, baseline=0]{\draw (0,0) node[tree]{} -- (-1,-1) node[tree]{}; \draw (0,0) -- (1,-1) node[tree]{};
\draw (1,-1) node[tree]{} -- (0,-2) node[tree]{}; \draw (1,-1) -- (1,-2) node[tree]{}; \draw (1,-1) -- (2,-2) node[tree]{};} 
\quad
\tikz[scale=\sf, baseline=0]{\draw (0,0) node[tree]{} -- (-1,-1) node[tree]{}; \draw (0,0) -- (0,-1) node[tree]{}; \draw (0,0) -- (1,-1) node[tree]{}; 
\draw (-1,-1) node[tree]{} -- (-2,-2) node[tree]{}; \draw (-1,-1) -- (0,-2) node[tree]{};} 
\quad
\tikz[scale=\sf, baseline=0]{\draw (0,0) node[tree]{} -- (-1,-1) node[tree]{}; \draw (0,0) -- (0,-1) node[tree]{}; \draw (0,0) -- (1,-1) node[tree]{}; 
\draw (0,-1) node[tree]{} -- (-1,-2) node[tree]{}; \draw (0,-1) -- (1,-2) node[tree]{};} 
\quad
\tikz[scale=\sf, baseline=0]{\draw (0,0) node[tree]{} -- (-1,-1) node[tree]{}; \draw (0,0) -- (0,-1) node[tree]{}; \draw (0,0) -- (1,-1) node[tree]{}; 
\draw (1,-1) node[tree]{} -- (0,-2) node[tree]{}; \draw (1,-1) -- (2,-2) node[tree]{};} 
} \\[28pt]
\multicolumn{3}{c}{
\tikz[scale=\sf, baseline=0]{\draw (0,0) node[tree]{} -- (-1.25,-1) node[tree]{}; \draw (0,0) -- (1.25,-1) node[tree]{};
\draw (-1.25,-1) -- (-2.25,-2) node[tree]{}; \draw (-1.25,-1) -- (-0.3,-2) node[tree]{};
\draw (1.25,-1) -- (0.3,-2) node[tree]{}; \draw (1.25,-1) -- (2.25,-2) node[tree]{};} 
\quad
\tikz[scale=\sf, baseline=0]{\draw (0,0) node[tree]{} -- (-1,-1) node[tree]{}; \draw (0,0) -- (1,-1) node[tree]{};
\draw (-1,-1) -- (-2,-2) node[tree]{}; \draw (-1,-1) -- (0,-2) node[tree]{};
\draw (-2,-2) node[tree]{} -- (-3,-3) node[tree]{}; \draw (-2,-2) -- (-1,-3) node[tree]{};} 
\quad
\tikz[scale=\sf, baseline=0]{\draw (0,0) node[tree]{} -- (-1,-1) node[tree]{}; \draw (0,0) -- (1,-1) node[tree]{};
\draw (-1,-1) -- (-2,-2) node[tree]{}; \draw (-1,-1) -- (0,-2) node[tree]{};
\draw (0,-2) node[tree]{} -- (-1,-3) node[tree]{}; \draw (0,-2) -- (1,-3) node[tree]{};} 
\quad
\tikz[scale=\sf, baseline=0]{\draw (0,0) node[tree]{} -- (-1,-1) node[tree]{}; \draw (0,0) -- (1,-1) node[tree]{};
\draw (1,-1) -- (0,-2) node[tree]{}; \draw (1,-1) -- (2,-2) node[tree]{};
\draw (0,-2) node[tree]{} -- (-1,-3) node[tree]{}; \draw (0,-2) -- (1,-3) node[tree]{};}
\quad
\tikz[scale=\sf, baseline=0]{\draw (0,0) node[tree]{} -- (-1,-1) node[tree]{}; \draw (0,0) -- (1,-1) node[tree]{};
\draw (1,-1) -- (0,-2) node[tree]{}; \draw (1,-1) -- (2,-2) node[tree]{};
\draw (2,-2) node[tree]{} -- (1,-3) node[tree]{}; \draw (2,-2) -- (3,-3) node[tree]{};}
} \\[42pt] \hline\hline
\end{tabular}
\bigskip
\caption{Elements of $\ST_{\!n}$ for $n=0,1,2,3$.}
\end{table}

Let $\T_n$ be the set of ordered rooted trees with $n$ generators. Our next result shows that every element of $\T = \bigcup_n \T_n$ can be uniquely constructed from elements of the set $\ST = \bigcup_n \ST_{\!n}$ via a Dyck path insertion algorithm.

\begin{theorem} \label{thm:dyck2tree}
The set of ordered rooted trees with $n$ generators is in one-to-one correspondence with the set of Dyck paths of semilength $n$ whose ascents of length $j$, for $j=1,\dots,n$, may be colored in $s_{j-1}$ different ways.
\end{theorem}

\begin{proof}
Let $\mathfrak{D}^{\mathbf{s}}_n(1,0)$ be the set of Dyck paths of semilength $n$ colored by the little Schr\"oder numbers.  Because of ({\sc tree}), we can consider $\mathfrak{D}^{\mathbf{s}}_n(1,0)$ to be the set of Dyck paths of semilength $n$ with each $j$-ascent associated with a tree in $\ST_{\!j-1}$. 

We will define a map $\psi: \mathfrak{D}^{\mathbf{s}}_n(1,0) \to \T_n$ that requires a notion of traversing an ordered tree. When traversing a tree, we stick to the ordering of leaf nodes that would be acquired when performing inorder depth-first search.
    
Suppose we have a path $P\in\mathfrak{D}^{\mathbf{s}}_n(1,0)$ with $k$ peaks. For $1\leq i \leq k$, let $G_i$ denote the ordered tree associated with the $i$th ascent of the Dyck path. Starting at the first ascent of $P$, we construct the ordered tree $\psi(P)$ as follows. We start at the root of the tree $G_1$. Suppose that the first ascent of the Dyck path is followed by a $j_1$-descent. Then, we traverse $G_1$ from the root until we reach the $j_1$th leaf. At that point, we add an edge between the $j_1$th leaf of $G_1$ and the root of the tree $G_2$ corresponding to the second ascent. Suppose that the second ascent is followed by a $j_2$-descent. From the node that was formerly the $j_1$th leaf from the root (and which is no longer a leaf), we continue to traverse the augmented tree until we reach the $j_2$th leaf from that node.
    
This algorithm will always result in an ordered tree with $n$ generators. Note that since $P$ has $n$ up steps, and since every $j$-ascent corresponds to an ordered tree with $j$ leaves (and no other generators), there are $n$ leaves total in the trees $G_1,\dots,G_k$. Every leaf in a tree $G_i$ either becomes a leaf in the final tree, or it becomes a node with a single child. Thus the final tree $\psi(P)$ has exactly $n$ generators. 

An example that illustrates the above construction is given in Figure~\ref{fig:dyck2tree}.

\begin{figure}[ht]
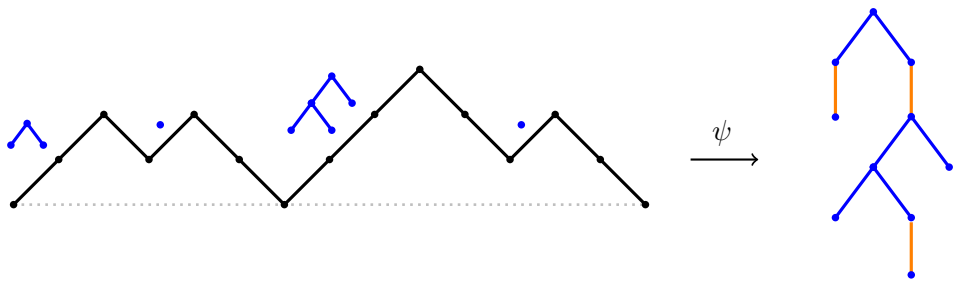

\tikzstyle{dot}=[circle, inner sep=1, fill]
\tikz[scale=0.6,baseline=-28pt]{
 \draw[pathlight] (0,0) -- (14,0);
 \draw[path] (0,0) -- (2,2) -- (3,1) -- (4,2) -- (6, 0) -- (9,3) -- (11,1) -- (12,2) -- (14,0);
 \foreach \i in {0,1,2} {\node[dot] at (\i,\i) {};}
 \foreach \i/\j in {3/1,4/2,5/1,6/0,7/1,8/2,9/3,10/2,11/1,12/2,13/1,14/0} {\node[dot] at (\i,\j) {};}
 \drawbranch{9.4}{3.8}{0.75}
 \drawbranch{8.8}{3}{0.75}
 \drawbranch{0.5}{3}{0.6}
 \draw[above=5pt,blue] (3.25,1.5) node[dot]{}; 
 \draw[above=5pt,blue] (11.25,1.5) node[dot]{}; 
 \draw[->,thick] (15,1) -- (16.5,1);
 \node[above=1pt] at (15.7,1) {\small $\psi$};}
\qquad
\tikz[scale=0.7]{
 \draw[very thick,orange] (-.72,-1) -- (-.72,-2);
 \draw[very thick,orange] (.72,-1) -- (.72,-2);
 \draw[very thick,orange] (.72,-4) -- (.72,-5);
 \drawbranch{0}{0}{1.2};
 \draw (-.72,-2) node[dot,blue]{};
 \drawbranch{0.6}{-1.66}{1.2};
 \drawbranch{0}{-2.46}{1.2};
 \draw (.72,-5) node[dot,blue]{};
}
\caption{Example for $n=7$.}
\label{fig:dyck2tree}
\end{figure}

The inverse map from $\T_n$ to $\mathfrak{D}^{\mathbf{s}}_n(1,0)$ is straightforward. Given $G\in \T_n$, we construct a Schr\"oder colored Dyck path as follows. Starting at the root of $G$, we first identify the largest subgraph of $G$ with same root that is an ordered tree with no vertex of outdegree 1. We call this subgraph $G_1$. Once $G_1$ has been specified, we count the number of leaves of $G_1$.  If $G_1$ has $j$ leaves, then we begin the Dyck path with a $j$-ascent and label that first ascent with $G_1$. Then, we traverse $G_1$ until we find a leaf of $G_1$ that corresponds to a node of outdegree 1 in $G$. If this leaf is the $j_1$th leaf checked, then we add a $j_1$-descent to the Dyck path. Then, to find the next ascent, we identify the ordered tree $G_2$ rooted at the child of the node corresponding to the $j_1$th leaf of $G_1$. This process continues until the tree $G$ has been completely traversed.
    
This algorithm gives a Dyck path in $\mathfrak{D}^{\mathbf{s}}_n(1,0)$. Indeed, every generator in $G$ is counted as both an ascent and a descent, so the path $\psi^{-1}(G)$ has semilength $n$. Moreover, every $j$-ascent corresponds to one of the $s_{j-1}$ possible ordered trees with $j$ leaves and no nodes of outdegree 1. Finally, note that the ascent corresponding to a generator of the ordered tree is always added to the path before the descent corresponding to the same generator is added to the path.
\end{proof}

\begin{corollary}
As a consequence of \cite[Theorem~2]{BGMW} and Theorem~\ref{thm:dyck2tree}, we get
\begin{equation*}
   \#\T_n = \#\mathfrak{D}^{\mathbf{s}}_n(1,0) = 1 + \sum_{k=1}^{n-1}\sum_{j=1}^{n-k}\frac{1}{j} \binom{n}{k-1}\binom{n-k-1}{j-1}\binom{n+j-1}{j-1}.
\end{equation*}
This gives the sequence $1, 2, 7, 32, 166, 926, 5419, 32816,\dots$, see \cite[A108524]{oeis}.

Moreover, the number of trees in $\T_n$ having exactly $k-1$ nodes of outdegree 1 is equal to the number of paths in $\mathfrak{D}^{\mathbf{s}}_n(1,0)$ having $k$ peaks. This number is given by
\begin{equation*}
   \binom{n}{k-1} \frac{(k-1)!}{n!} B_{n,k}(1!s_0,2!s_1,\dots),
\end{equation*}
which for $1\le k<n$ can be written as
\begin{equation*}
   \binom{n}{k-1} \sum_{j=1}^{n-k}\frac{1}{j}\binom{n-k-1}{j-1}\binom{n+j-1}{j-1}.
\end{equation*}
\end{corollary}

\section{Outerplanar maps}
\label{sec:OuterplanarMaps}

Recall that if $\mathbf{s}$ is the sequence $(s_n)_{n\in\mathbb{N}_0}$, $\mathfrak{D}^{\mathbf{s}}_n(1,1)$ denotes the set of Dyck paths of semilength $2n$ created from down-steps and blocks of the form $P_j=``u^{2j}d^{j}"$ for $j=1,\ldots,n$, such that each $2j$-ascent may be colored in $s_{j-1}$ different ways.

In this section, we use the interpretation ({\sc poly}) for the little Schr\"oder numbers to generate and count simple outerplanar maps using polygon dissections (seen as biconnected outerplanar maps) put together according to the elements of $\mathfrak{D}^{\mathbf{s}}_n(1,1)$.

Let $\B_n$ be the set of biconnected (nonseparable) rooted outerplanar maps with $n+2$ vertices. Note that $\B_0$ consists of the single map \tikz{\draw[poly] (0,0) node[v]{} -- (0.7,0) node[v]{};}, and for $n\ge1$, every element of $\B_n$ can be seen as a dissection of a convex $(n+2)$-gon by nonintersecting diagonals. Therefore, $\#\B_n = s_{n}$.
For example, $\B_1$ consists of a single triangle and $\B_2$ contains the three possible dissections of a square:

\medskip
\begin{center}
\tikz{\node[poly, regular polygon sides=3, minimum size=40] at (0,0) (A) {};
\node[regular polygon, regular polygon sides=3, minimum size=39] at (0,0) (B) {};
\foreach \i in {1,...,3}{\node[v] at (B.corner \i) {};}}
\quad\;
\tikz{\draw[dotted,thick,gray] (0,0) -- (0,1.25)}
\quad\;
\tikz{\node[poly, regular polygon sides=4, minimum size=40] at (0,0) (A) {};
\node[regular polygon, regular polygon sides=4, minimum size=39] at (0,0) (B) {};
\foreach \i in {1,...,4}{\node[v] at (B.corner \i) {};}}
\quad
\tikz{\node[poly, regular polygon sides=4, minimum size=40] at (0,0) (A) {};
\node[regular polygon, regular polygon sides=4, minimum size=39] at (0,0) (B) {};
\draw[poly] (B.corner 1) -- (B.corner 3);
\foreach \i in {1,...,4}{\node[v] at (B.corner \i) {};}}
\quad
\tikz{\node[poly, regular polygon sides=4, minimum size=40] at (0,0) (A) {};
\node[regular polygon, regular polygon sides=4, minimum size=39] at (0,0) (B) {};
\draw[poly] (B.corner 2) -- (B.corner 4);
\foreach \i in {1,...,4}{\node[v] at (B.corner \i) {};}}
\end{center} 

By default, we will assume that the elements of $\B_n$ are rooted at the base of the polygon such that the root face is the outer face.

\begin{theorem} \label{thm:map2dyck}
The set $\M_n$ of simple outerplanar maps with $n+1$ vertices is in one-to-one correspondence with the set $\mathfrak{D}^{\mathbf{s}}_n(1,1)$. 
\end{theorem}
\begin{proof}
We will establish a bijective map $\phi:\M_n \to \mathfrak{D}^{\mathbf{s}}_n(1,1)$ that relies on the decomposition of a planar map into biconnected components. Given a map $M\in\M_n$ with $k$ biconnected components, we construct a Dyck path $P=\phi(M)$ as follows. Starting at the root vertex $v_0$ of $M$, we walk counterclockwise around $M$, visiting all of its vertices. We let $B_1,\dots,B_k$ be the biconnected components of $M$, numbered in the order they are visited, and assume that $B_j$ has $i_j+1$ vertices, $1\le i_j\le n$ for every $j$. While at $v_0$, add a $2i_1$-ascent to $P$ followed by $i_1$ down-steps and walk to the next vertex of $M$ (second vertex of the root edge). For each subsequent vertex we visit on $B_1$, add a down-step to $P$ until we reach the first vertex on $B_2$. At that point, we add a $2i_2$-ascent to $P$ followed by $i_2$ down-steps, and move to the next vertex. We then continue adding down-steps to $P$ until we reach the first vertex of the next component of $M$. When visiting a vertex of a biconnected component that has already been visited, we add a down-step to $P$. This process ends after we have visited all vertices of $M$ and reach $v_0$ again. We then label/color the $k$ maximal ascents of $P$ with the corresponding biconnected maps $B_1,\dots,B_k$, and obtain an element of $\mathfrak{D}^{\mathbf{s}}_n(1,1)$ with $k$ peaks. An example of the above construction is shown in Figure~\ref{fig:map2dyck}.

\begin{figure}[ht]
 \tikz[scale=0.5]{
 \node[poly, regular polygon sides=4, minimum size=40] at (0,0) (A) {};
 \node[regular polygon, regular polygon sides=4, minimum size=39] at (0,0) (B) {};
 \draw[poly] (B.corner 1) -- (B.corner 3);
 \foreach \i in {1,...,4}{\node[v] at (B.corner \i) {};}
 \draw[poly] (B.corner 4) node[v]{} -- +(0,-56pt) node[v]{};
 \draw[poly] (B.corner 2) node[v]{} -- +(-30pt,50pt) node[v]{} -- +(-60pt,0pt) node[v]{} -- cycle;
 \draw[-stealth, very thick] (B.corner 3)+(27pt,0)  -- +(35pt,0);
 \draw[->, thick] (2,0) -- (3.5,0);
 \node[above=1pt] at (2.7,0) {\small $\phi$};
}
\tikzstyle{dot}=[circle, inner sep=1, fill=black]
\quad \tikz[scale=0.4,baseline=-10pt]{
 \draw[pathlight] (0,0) -- (24,0);
 \draw[path] (0,0) -- (6,6) -- (9,3) -- (11,5) -- (15, 1) -- (19,5) -- (24,0);
 \foreach \i in {0,...,6} {\node[dot] at (\i,\i) {};}
 \foreach \i/\j in {7/5,8/4,9/3,10/4,11/5,12/4,13/3,14/2,15/1,16/2,17/3,18/4,19/5,20/4,21/3,22/2,23/1,24/0} 
 {\node[dot] at (\i,\j) {};}
 \node[above=15pt] at (1.5,0.5) 
 {\tikz{
 \node[minipoly, regular polygon sides=4, minimum size=27] at (0,0) (A) {};
 \node[regular polygon, regular polygon sides=4, minimum size=26] at (0,-1pt) (B) {};
 \draw[minipoly] (A.corner 1) -- (A.corner 3);
 \foreach \i in {1,...,4}{\node[v0] at (B.corner \i) {};}
 }};
 \node[above=15pt] at (9.2,2) 
 {\tikz{\draw[minipoly] (0,0) -- (0.5,0); \node[v0] at (0,-0.05){}; \node[v0] at (0.5,-0.05){};}};
 \node[above=15pt] at (16,1) 
 {\tikz{
 \node[minipoly, regular polygon sides=3, minimum size=23] at (0,0) (A) {};
 \node[regular polygon, regular polygon sides=3, minimum size=22] at (0,-1pt) (B) {};
 \foreach \i in {1,...,3}{\node[v0] at (B.corner \i) {};}
 }}
}
\caption{Example for $n=6$ and $k=3$.}
\label{fig:map2dyck}
\end{figure}

The inverse map is clear. Given a Dyck path $P\in \mathfrak{D}^{\mathbf{s}}_n(1,1)$ with $k$ maximal ascents of lengths $2i_1,\dots,2i_k$, colored by biconnected rooted maps $B_1,\dots,B_k$, we construct an outerplanar map as follows. Let $M_1=B_1$ and let $v_0$ be the root vertex of $B_1$. If the first ascent of $P$ is followed by $j_1$ down-steps, we walk around $M_1$ (starting at $v_0$) until we reach the $(j_1-i_1+1)$th vertex on $M_1$. We call that vertex $v_1$ and merge it with the root vertex of $B_2$ to create a rooted outerplanar map $M_2$ with the same root edge as $M_1$. This extended map has now $i_1+ i_2 +1$ vertices. We then repeat the process starting now at $v_1$. After all coloring biconnected maps have been used as extensions (traversing $P$ from left to right), the resulting rooted outerplanar map $M_k$ will be an element of $\M_n$. We let $\phi^{-1}(P)=M_k$.
\end{proof}

\begin{corollary}
By Theorem~\ref{thm:map2dyck}, the number of outerplanar maps in $\M_n$ having exactly $k$ biconnected components is equal to the number of paths in $\mathfrak{D}^{\mathbf{s}}_n(1,1)$ having $k$ peaks. Thus, by \cite[Theorem~2]{BGMW}, this number is given by
\begin{equation}\label{eq:biconn_components}
   \binom{n+k}{k-1} \frac{(k-1)!}{n!} B_{n,k}(1!s_0,2!s_1,\dots),
\end{equation}
which for $1\le k<n$ can be written as
\begin{equation*}
   \binom{n+k}{k-1} \sum_{j=1}^{n-k}\frac{1}{j}\binom{n-k-1}{j-1}\binom{n+j-1}{j-1}.
\end{equation*}
Note that if $k=n$, then \eqref{eq:biconn_components} becomes $\frac{1}{n}\binom{2n}{n-1} = C_n$, and we recover the known fact that planted trees are counted by the Catalan numbers.
\end{corollary}

\begin{corollary}
For $n\ge 1$ we have
\begin{equation*}
   \#\M_n = \frac{1}{n}\binom{2n}{n-1} + \sum_{k=1}^{n-1}\sum_{j=1}^{n-k}\frac{1}{j} \binom{n+k}{k-1}\binom{n-k-1}{j-1}\binom{n+j-1}{j-1}.
\end{equation*}
This gives the sequence $1, 3, 13, 67, 381, 2307, 14589, 95235,\dots$, see \cite[A064062]{oeis}.
 \end{corollary}


\end{document}